\documentclass[12pt]{article}

\usepackage{amsmath,amsfonts,amssymb,amsthm,amsrefs}
\usepackage{enumitem}

\usepackage{color}
\usepackage[dvipsnames]{xcolor}
\usepackage[normalem]{ulem}

\title{Combinatorial principles equivalent to weak induction}
\author{Caleb Davis \and Denis R.~Hirschfeldt \and Jeffry L.~Hirst \and Jake Pardo \and Arno Pauly \and Keita Yokoyama}

\date{December 24, 2018}%Submitted version

\setlist[enumerate]{label=\rm{(\arabic*)}, ref=\arabic*}

\theoremstyle{plain}
\newtheorem{thm}{Theorem}
\newtheorem{lemma}[thm]{Lemma}

\theoremstyle{definition}
\newtheorem*{sert}{$\ert$}
\newtheorem*{sect}{$\ect$}
\newtheorem*{ssrt}{$\srt$}

\newcommand{\nat}{\mathbb N}
\newcommand{\ert}{{\sf{ERT}}}
\newcommand{\ect}{{\sf{ECT}}}
\newcommand{\srt}{{\sf{SRT}^2_2}}
\newcommand{\rca}{{\sf{RCA}}_0}
\newcommand{\rcas}{{\sf{RCA}}_0^*}
\newcommand{\sfi}{{\sf I}}
\newcommand{\szz}{\Sigma^0_0}
\newcommand{\szo}{\Sigma^0_1}
\newcommand{\szt}{\Sigma^0_2}
\newcommand{\szr}{\Sigma^0_3}
\newcommand{\pzo}{\Pi^0_1}
\newcommand{\pzt}{\Pi^0_2}
\newcommand{\pzr}{\Pi^0_3}
\newcommand{\lh}{{\rm{length}}}
\newcommand{\lew}{\leq_{\rm W}}
\newcommand{\ltw}{<_{\rm W}}
\newcommand{\eqw}{\equiv_{\rm W}}
\newcommand{\lpo}{{\sf{LPO}}}
\newcommand{\wcn}{{\sf{C}}_\nat}
\newcommand{\wtcn}{{\sf{TC}}_\nat}
\newcommand{\lesw}{\leq_{\rm sW}}
\newcommand{\ltsw}{<_{\rm sW}}
\newcommand{\minert}{{\sf{min}\ert}}
\newcommand{\minect}{{\sf{min}\ect}}
\newcommand{\isinf}{{\sf{isInfinite}}}
\newcommand{\cmh}{{\sf{C}^\#_{\rm{max}}}}

\begin{document}

\maketitle

\begin{abstract}
We consider two combinatorial principles, $\ert$ and $\ect$.  Both
are easily proved in $\rca$ plus $\szt$ induction.  We give two proofs of
$\ert$ in $\rca$, using different methods to eliminate the use of $\szt$ induction.
Working in the weakened base system $\rcas$, we prove that $\ert$ is
equivalent to $\szo$ induction and $\ect$ is equivalent to $\szt$ induction.
We conclude with a Weihrauch analysis of the principles, showing
$\ert \eqw \lpo^* \ltw \wtcn^* \eqw \ect$.
\end{abstract}

In their logical analysis of vertex colorings of hypergraphs,
Davis, Hirst, Pardo, and Ransom \cite{davisetal} isolate
the combinatorial principle $\ert$, and relate it to the nonexistence
of finite conflict-free colorings for a particular hypergraph.
The principle asserts that for any finite coloring of the natural numbers $\nat$
there is a tail of the coloring such that every color appearing in the tail
appears at least twice in the tail.  $\ert$ stands for {\emph e}ventually
{\emph r}epeating {\emph t}ail, and can be formulated as follows.

\begin{sert}
If $f: \nat \to k$ for some $k \in \nat$,
then there is a $b \in \nat$ such that for all $x \ge b$, there
is a $y \ge b$ such that $x\neq y$ and $f(x) = f(y)$.
\end{sert}

The principle $\ert$ is an immediate consequence of
the principle $\ect$ introduced by Hirst \cite{hirst}.
$\ect$ stands for {\emph e}ventually {\emph c}onstant palette {\emph t}ail,
and asserts that for any finite coloring of $\nat$ there is a tail of the coloring
such that the colors appearing in any final segment of the tail are exactly
those appearing in the entire tail.  A more formal version follows.

\begin{sect}
If $f: \nat \to k$ for some $k \in \nat$,
then there is a $b \in \nat$ such that for all $x \ge b$, there
is a $y >x$ such that $f(x) = f(y)$.
\end{sect}

Both Davis et al.~\cite{davisetal} and Hirst \cite{hirst} work in the usual framework
of reverse mathematics.  In particular, they prove equivalences over the subsystem
of second order arithmetic $\rca$.  This axiom system includes basic natural number
arithmetic axioms, an induction scheme restricted to $\szo$ formulas (denoted $\sfi\szo$),
and a recursive comprehension axiom that essentially asserts that computable %D:  removed the
sets of natural numbers exist.  See Simpson's book \cite{simpson} for more about %D:  change subsets to sets
$\rca$ and reverse mathematics.  Theorem~6 of Hirst \cite{hirst} shows that over $\rca$,
$\ect$ is equivalent to $\sfi\szt$, an induction scheme for $\szt$ formulas.
$\rca$ can prove that $\ect$ implies $\ert$, so $\rca$ proves that $\sfi\szt$ implies $\ert$.
As we will see in the next section, $\sfi\szt$ is not necessary in this proof.

\subsubsection*{1.  $\rca$ proves $\ert$}

Davis et al.~\cite{davisetal} show that $\sfi\szt$ is not needed in the proof of $\ert$
%D: Added ``of'' below
by deriving $\ert$ from a restricted form of Ramsey's theorem and applying a
%K:  deleted "conservation"
result of Chong, Slaman, and Yang \cite{csy}.  There, Ramsey's theorem
is restricted to stable colorings of pairs, that is to functions $f: [\nat ]^2 \to k$ such that
for all $x$, $\lim_{y \to \infty} f(x,y)$ exists.  Stable Ramsey's theorem for pairs and
two colors is denoted by $\srt$ and can be formalized as follows.

\begin{ssrt}
If $f : [\nat ]^2 \to 2$ is stable, then there is an infinite set $H \subset \nat$ and
a color $c\in \{0,1\}$ such that for all $(x,y) \in [H]^2$, $f(x,y) = c$.
\end{ssrt}

The next result appears as Theorem~11 in Davis et al.~\cite{davisetal}.  The $\rca$ in
parentheses indicates that the proof can be carried out in the formal system $\rca$.  For completeness,
we give a minimal sketch of the proof.

\begin{lemma}\label{h1}
$(\rca)$  $\srt$ implies $\ert$.
\end{lemma}

\begin{proof}
  Working in $\rca$, let $f: \nat \to k$ be a coloring of $\nat$ as in the statement of $\ert$.  Define
%D: Changed ``an'' to ``and'' below  
a coloring of pairs, $g:[\nat]^2 \to 2$ by $g(a,b) = 1$ if and only if for some $x$ in the half open
interval of natural numbers $[a,b)$, $f(x)$ appears exactly once in the range of $f$ restricted to
$[a,b)$.  Because $g$ is stable, by $\srt$ there is an infinite homogeneous set $H$.  An argument based on
the first $3\cdot 2^{k-1} $ elements of $H$ shows that $g$ is identically equal to $0$.  Consequently,
the minimum element of $H$ satisfies the requirements of the bound $b$ in the statement of $\ert$.
(For a more complete proof, see Davis et al.~\cite{davisetal}).
\end{proof}

By Corollary 2.6 of Chong, Slaman, and Yang \cite{csy}, $\srt$ cannot prove $\sfi\szt$, so neither
can $\ert$.  Thus although $\rca + \sfi\szt$ proves $\ert$, the full strength of $\sfi\szt$ is not
necessary.  Using a recent conservation result of Patey and Yokoyama \cite{py}, together
with an alternative formalization of $\ert$, we can show that $\rca$ proves $\ert$, completely
eliminating the use of $\sfi\szt$.

\begin{lemma}\label{k1}
$(\rca )$  The following are equivalent.
\begin{enumerate}
\item \label{k1-1}$\ert$.
\item \label{k1-2}$\ert^\prime:$  If $f: \nat \to k$ for some $k \in \nat$, then there is a number $b \in \nat$,
a set $I \subset [0,k)$ consisting of the range of $f$ on $[b, \infty)$, and a witness set
%D: Added ``we have'' for clarity      
$\{ ( x_i , y_i ) \mid i \in I \}$ such that for every $z \ge b$, we have
$f(z) \in I$, $b \le x_{f(z)} < y_{f(z)}$, and $f(z) = f(x_{f(z)} ) = f(y_{f(z)})$.
\end{enumerate}
\end{lemma}

\begin{proof}
We will work in $\rca$.
Note that for any $f$, the number $b$ provided by $\ert^\prime$ also satisfies the statement of $\ert$.
Thus $\ert$ follows immediately from $\ert^\prime$.

To prove the converse, let $f: \nat \to k$ and apply $\ert$ to obtain $b$.  The set
$I = \{j<k \mid \exists t ( t \ge b \land f(t) = j) \}$ exists by bounded $\szo$ comprehension,
a consequence of $\rca$ \cite{simpson}*{Theorem~II.3.9}.  For each $i \in I$, there are at least two distinct
values $x_i \ge b$ and $y_i \ge b$ such that $f(x_i ) = f(y_i ) = i$.  Picking the least such witness pair for each $i$,
recursive comprehension proves the existence of the witness set $\{ ( x_i , y_i ) \mid i \in I\}$.
Routine arguments verify that $b$ and this witness set satisfy the requirements of $\ert^\prime$.
\end{proof}

Applying the two lemmas and using a result of Patey and Yokoyama \cite{py}, we can easily prove
$\ert$ in $\rca$, answering a question of Davis et al.~\cite{davisetal}.
An alternative direct proof of Theorem~\ref{k2} is included in the next section in the proof of Theorem~\ref{D3}.  %D:  forward reference

\begin{thm}\label{k2}
$\rca$ proves $\ert$.
\end{thm}

\begin{proof}
$(\rca )$  By Lemma~\ref{h1}, $\rca + \srt$ proves $\ert$.  Thus, by Lemma~\ref{k1},
$\rca + \srt$ proves $\ert ^\prime$.  By Theorem~7.4 of Patey and Yokoyama \cite{py},
$\rca + \srt$ is a conservative extension of $\rca$ for formulas of the form
$\forall X \varphi (X)$, where $\varphi$ is $\pzr$.  $\ert^\prime$ has this form, so $\rca$
proves $\ert^\prime$.  By Lemma~\ref{k1}, $\rca$ proves $\ert$.
\end{proof}

The conservation result of Patey and Yokoyama is a powerful tool for eliminating the
%D: Added ``of'' below (after ``use'').
use of $\szt$ induction in the proofs of combinatorial principles.  Their result actually holds
for Ramsey's theorem for pairs and two colors, so it is not necessary to limit ourselves
to stable colorings.  The principle $\ert^\prime$ can be formalized in the form $\forall X \varphi(X)$
where $\varphi$ is $\szt$.  Clearly, we have made use of less than the full strength of this technique
in our example.  On the other hand, if $\ert$ is directly formalized in the form $\forall X \theta (X)$,
the formula $\theta$ is $\szr$, so Patey and Yokoyama's result does not apply.  Lemma~\ref{k1} is
a necessary step in the argument.

\subsubsection*{2.  Reverse mathematics:  $\ert$ is $\sfi\szo$ and $\ect$ is $\sfi\szt$}

In this section, we prove that our combinatorial principles are equivalent to induction schemes
over the weakened base system $\rcas$.  The axioms of $\rcas$ are those of $\rca$ less the
$\szo$ induction scheme, with the addition of a $\szz$ induction scheme and function symbols and axioms for
integer exponentiation.  The subsystem is described in Chapter~X of Simpson's book \cite{simpson}.
The following lemma incorporates results from an early work of Simpson and Smith \cite{ss}.  Note the
change in the base system at the beginning of the statement of the lemma.

\begin{lemma}\label{D1}
$(\rcas )$ The following are equivalent.
\begin{enumerate}
\item  $\sfi\szo$, the $\szo$ induction scheme.\label{D1-1}
\item  The universe of total functions is closed under primitive recursion.\label{D1-2}
\item  Bounded $\szo$ comprehension.\label{D1-3}
\item  Bounded $\pzo$ comprehension.\label{D1-4}
\end{enumerate}
\end{lemma}

\begin{proof}
The equivalence of items (\ref{D1-1}), (\ref{D1-2}), and (\ref{D1-3}) are included in Lemma~2.5 of the article of
Simpson and Smith \cite{ss}.  Recursive comprehension proves the existence of complements of sets, so items
(\ref{D1-3}) and (\ref{D1-4}) are also equivalent.
\end{proof}

For our arguments, it is useful to formalize the concept of a partial function.  Working in $\rcas$, we can define
a code for a finite partial function as a set of ordered pairs $f \subset [0,k)\times \nat$ such that for all $i$, $n$, and $m$,
if $(i,n)\in f$ and $(i,m) \in f$, then $n=m$.  Using this notion, we can state another equivalent form of $\sfi\szo$.

\begin{lemma}\label{D2}
$(\rcas )$ The following are equivalent:
\begin{enumerate}
\item \label{D2-1} $\sfi\szo$.
\item \label{D2-2} Finite partial functions have bounded ranges.  That is, if $f\subset k \times \nat$ is a finite partial
function, then
\[
\exists b\, \forall i<k\, \forall n ((i,n) \in f \to n \le b ).
\]
\end{enumerate}
\end{lemma}

\begin{proof}
To prove (\ref{D2-1}) implies (\ref{D2-2}), working in $\rcas$, assume $\sfi\szo$ and let $f$ be a finite partial function
contained in $k \times \nat$.  By Lemma~\ref{D1}, we may apply bounded $\szo$ comprehension and find the
set $D = \{ x< k \mid \exists y (x,y) \in f\}$.  By recursive comprehension, there is a total function
$f^\prime$ satisfying
\[
f^\prime (n)= 
\begin{cases}
\min\{ m \mid (n,m) \in f \} &\text{if}~n\in D\\
0&\text{otherwise.}
\end{cases}
\]
By Lemma~\ref{D1}, we may apply primitive recursion to find the summation function
$g(n) = \sum_{i=0}^n f^\prime (i)$.  The integer $g(k-1)$ is a suitable bound for the range of $f$.

To prove the converse, we will use (\ref{D2-2}) to prove bounded $\szo$ comprehension.  Let $\theta (m,n)$
be a $\szz$ formula and fix a bound $k$.  We will prove that the set $\{ m<k \mid \exists n \theta (m,n)\}$ exists.
Using recursive comprehension, we can find the set of pairs
\[
%D: Changed ``m.n'' to ``m,n''
f = \{ (m,n) \mid \theta (m,n) \land \forall y<n \, \neg \theta (m,y ) \}.
\]
Note that $f$ is a partial function from $k$ into $\nat$.  By (\ref{D2-2}), there is a bound $b$ for the range of $f$.
Thus, for all $m<k$, $\exists n \theta(m,n)$ if and only if $\exists n\le b \, \theta (m,n)$.  So
$\{m<k \mid \exists n \theta (m,n) \}$ is identical to $\{ m<k \mid \exists n\le b\, \theta (m,n)\}$, and its existence follows from
recursive comprehension.
\end{proof}

We can now show that $\ert$ is equivalent to $\sfi\szo$ over $\rcas$.  Because $\rcas$ plus $\sfi\szo$ is $\rca$, this
provides a direct proof of $\ert$ in $\rca$, without the use of conservation results.
Following the proof of the theorem, we will comment on this as
a technique for eliminating $\sfi\szt$ in proofs of combinatorial results.

\begin{thm}\label{D3}
$(\rcas )$  The following are equivalent.
\begin{enumerate}
\item\label{D3-1}  $\sfi\szo$.
\item\label{D3-2}  $\ert$.
\item\label{D3-3}  $\forall j \ert(j)$.  Here $\ert(j)$ generalizes $\ert$, requiring that at or after the bound $b$, any value
of $f$ that appears must
appear at least $j$ times.
\end{enumerate}
\end{thm}

\begin{proof}
To show that (\ref{D3-1}) implies (\ref{D3-2}), we could simply cite Theorem~\ref{k2}. %D:  Added motivation for direct proof.
We present a direct proof using sequential applications of bounded comprehension that will be
adapted to prove Theorem~\ref{W1} below.
Working in $\rcas$, assume $\sfi\szo$.  By Lemma~\ref{D1},  we may apply bounded $\szo$ comprehension.
We will prove $\ert$ for $f: \nat \to k$.  By bounded $\szo$ comprehension, we can find the set of (codes for) non-repeating
finite sequences of values less than $k$ such that the colors appear in this order somewhere in the range of $f$.  More formally,
bounded $\szo$ comprehension proves the existence of a set $S$ of (codes for) sequences such that $\sigma \in S$ if and only if
\begin{list}{$\bullet$}{}
\item $\lh (\sigma ) < k$,
\item  $\forall i < \lh (\sigma ) \, (\sigma (i) < k)$,  %D:  Added ( )
\item $\forall i < \lh (\sigma ) \, \forall j < \lh (\sigma) \, (\sigma (i ) = \sigma(j) \to i=j)$,
\end{list}
and there is a finite witness sequence $\tau$ such that
\begin{list}{$\bullet$}{}
\item  $\lh (\sigma ) = \lh (\tau )$,
\item $\forall i < \lh (\tau ) \, \forall j < \lh (\tau) \, (i<j \to \tau(i)<\tau (j))$,
\item $\forall i < \lh (\tau) (f(\tau(i)) = \sigma (i))$.
\end{list}
By Lemma~\ref{D1}, we may also use bounded $\pzo$ comprehension.  Using $S$ as a parameter and applying
bounded $\pzo$ comprehension, we can find a subset $T$ of $S$ consisting of the empty sequence and all those
sequences $\sigma$ such that the first time the colors appear in the specified order, the last color never reappears.
When selecting the first witness sequence, we assume that for sequences differing in a single entry, the
sequence with the smaller entry appears first.
Thus, $\sigma$ is in $T$ if and only if $\sigma $ is empty, or $\sigma \in S$ and for the first witness sequence $\tau$ for $\sigma$ and
any $j > \lh(\tau ) -1$, $f(j) \neq \sigma(\lh (\tau ) - 1)$.  $T$ is a subset of the finite set of non-repeating sequences of numbers less
than $k$, so $\rcas$ can answer questions about whether or not sequences are in $T$.  In particular,
we can define a subset $T_0 \subset T$ of sequences $\sigma$ such that no extension of $\sigma$ is in $T$
and every initial segment of $\sigma$ is in $T$.  Suppose $\sigma_0\in T_0$.
If $\sigma_0$ is empty, then
every color in the range of $f$ appears at least twice, and $b=0$ is the desired bound for $\ert$.  If $\sigma_0$ is nonempty,
let $\tau_0$ be the first witness sequence for $\sigma_0$, and define $b = \tau_0 (\lh (\sigma_0 )-1) + 1$.  Because $\sigma_0$ is in $T$ and
%D:  Added "in T"
none of its extensions are,
every color appearing at or after $b$ must appear at least twice.  Summarizing,
the bound $b$ satisfies the requirements of $\ert$.

%D: Added parentheses below
Next, we will show that (\ref{D3-2}) implies (\ref{D3-1}) by proving the contrapositive.  Suppose $\rcas$ holds and
$\sfi\szo$ fails.  By Lemma~\ref{D2}, there is a finite partial function $g \subset k \times \nat$ with an unbounded range.
Define the function $f : \nat \to k+1$ by
\[
f(n)=
\begin{cases}
j&{\text {if~}} j<k \land (j,n)\in g\\
k&\text{otherwise.}
\end{cases}
\]
The function $f$ exists by recursive comprehension, and for any $b$ there is an $n>b$ such that $f(n)<k$ and the value of $f(n)$
appears only once in the range of $f$.  Thus no $b$ can be a bound for $\ert$ applied to $f$, and $\ert$ fails.

Because (\ref{D3-2}) is a special case of (\ref{D3-3}), to complete the proof of the theorem, it suffices to show in $\rcas$ that
$\forall j \ert(j)$ follows from $\ert$.  By our previous work, $\sfi\szo$ follows from $\ert$, so we may work in $\rca$.
Fix $j$ and suppose $f: \nat \to k$.  Our goal is to find a $b$ such that every color appearing at or after 
$b$ appears at least $j$ times in the range of $f$ at or after $b$.  Define $g: \nat \to k\times j$ by setting
\[
g(n) = ( f(n) , {\text{mod}}_j |\{ i< n \mid f(i) = f(n) \}|).
\]
Intuitively, if $f$ takes the value $i$ at locations $x_0, x_1, \dots x_j$ (and nowhere before or in between),
%D: Added commas around dots below
then $g(x_0) = (i,0)$, $g(x_1) = (i,1),$ $\dots,$ $g(x_{j-1}) = (i,j-1)$,
and $g(x_j) = (i,0)$.  Using a bijection between $k\times j$ and the natural numbers less than $k\cdot j$, we can view $g$ as a function
from $\nat$ into $k\cdot j$.
Let $b$ be a bound for $\ert$ applied to $g$.  Suppose color $i$ appears at or after $b$ in the range of $f$.  Let $x_0$ be the first
such location.  Then for some $m<j$, $g(x_0) = (i,m)$.  Note that $x_0$ is the first location at or after $b$ where
$g$ takes this value.  By $\ert$ for $g$, there is an $x_1>x_0$ such that $g(x_1)= g(x_0)$.  By the definition of $g$,
there are at least $j$ places in $[x_0, x_1 )$ where $f$ takes the value $i$.  Thus $b$ is a bound for $\ert(j)$ for $f$.
This completes the proof of (\ref{D3-3}) from (\ref{D3-2}) and the proof of the theorem.
\end{proof}

For use in the proof of Theorem~\ref{minert}, note that the set $T$ defined in the preceding proof can be used
to compute the minimum bound satisfying $\ert$.  Because we are making a computability theoretic argument,
we are not restricted to $\rca^*$.  If every color in the range of $f$ appears at least twice, then no sequences
of length one appear in $T_0$, so $\sigma_0$ is the empty sequence and $b=0$ is the minimum bound.
Otherwise, define finite sequences $\sigma$ and $\tau$ as follows.  Let $\sigma(0)$ be the
last appearing among colors that appear exactly once, and let $\tau(0)$ be the location where $\sigma(0)$ appears.
Let $\sigma(i+1)$ be the last appearing among colors that appear exactly once after $\tau(i)$ if such a color exists,
and let $\tau(i+1)$ be the last location where $\sigma(i)$ appears.  If no such color exists, terminate the sequences.
Routine verifications show that $\sigma \in T_0$ and that $\tau$ is the first witness for $\sigma$, so that
$b=\tau(\lh (\sigma )-1)+1$ is a bound for $\ert$.  From the construction, if $b^\prime$ is any bound for $\ert$,
then $b^\prime > \tau (0)$, and for $i<\lh (\sigma )$, if $b^\prime >\tau(i)$ then $b^\prime > \tau (i+1)$.
Thus $b$ is minimal.  Consequently, the minimum bound can be calculated by listing $T_0$, calculating
the value $b$ for each sequence in $T_0$, and then selecting the minimum bound.

The existence of the set $T$ in the proof that (\ref{D3-1}) implies (\ref{D3-2}) above used an application of
bounded $\szo$ comprehension followed by an application of bounded $\pzo$ comprehension.
Na\"\i{}vely concatenating the associated formulas to construct $T$ with a single application results
in a use of bounded $\szt$ comprehension, a principle equivalent to $\sfi\szt$ \cite{simpson}*{Exercise~II.3.13}.  Conversely, it may be
possible to eliminate unnecessary uses of $\sfi\szt$ in proofs, particularly in the guise of
%D: Added ``of'' below (after ``applications'')
bounded $\szt$ or bounded $\pzt$ comprehension, by using a sequence of applications of bounded $\szo$ or bounded $\pzo$
comprehension.  In the case of the preceding proof, the sequential applications can be
combined into a single application, as in the second part of the proof of Theorem~\ref{W1} below.

%D:  Added comments on similarities to old results.
We complete this section with a proof of the equivalence of $\sfi\szt$ and $\ect$, showing
that $\ert$ is strictly weaker than $\ect$ over $\rcas$.  This result differs
from those in the article of Hirst \cite{hirst} in the use of the weaker base system $\rcas$.
The arguments here sidestep the tree colorings
used for \cite{hirst}*{Theorem~6} and in the alternative
argument following \cite{hirst}*{Theorem~7}, which is based
on the conservation result of Corduan, Groszek, and Mileti \cite{cgm}.

\begin{thm}\label{D4}
$(\rcas )$  The following are equivalent.
\begin{enumerate}
\item\label{D4-1} $\sfi\szt$.
\item\label{D4-2} $\ect$.
\end{enumerate}
\end{thm}

\begin{proof}
To prove that (\ref{D4-1}) implies (\ref{D4-2}), assume $\sfi\szt$ and fix $f: \nat \to k$.
Because $\sfi\szt$ implies $\sfi\szo$, we may work in $\rca$.  By bounded $\pzt$ comprehension,
a consequence of $\sfi\szt$ (\cite{simpson}*{Exercise~II.3.13}, plus complementation via
recursive comprehension), the set
\[
T = \{ j<k \mid \forall n \exists x (x>n \land f(x)= j )\}
\]
exists.  If $j\notin T$, then after some point $j$ ceases to appear in the range of $f$.  Formally,
\[
\forall j<k \, \exists s \forall x ((j \notin T \land x>s)\to f(x) \neq j ).
\]
By the $\pzo$ bounding principle, a consequence of $\sfi\szt$ \cite{simpson}*{Exercise~II.3.15},
there is a $b$ such that
\[
%D: Removed ``\exists s\le b\,'' below
\forall j<k \, \forall x ((j \notin T \land x>b)\to f(x) \neq j ).
\]
In particular, if $j \notin T$ then for all $x \ge b$ we have $f(x) \neq j$.  Summarizing, the range
of $f$ at or after $b$ is exactly $T$, and every value of $T$ appears infinitely often in the the range.
Thus $b$ satisfies the requirements of $\ect$.

We will prove that
%K:  deleted: prove that
(\ref{D4-2}) implies (\ref{D4-1}), by a two stage bootstrapping
argument.  For the first step, working in $\rcas$, note that $\ect$ implies $\ert$.  By Theorem~\ref{D3},
%D: Added period below
we may deduce $\sfi\szo$, so from now on we can work in $\rca$.

For the second step, we will use $\ect$ to prove bounded $\pzt$ comprehension, and then
deduce $\sfi\szt$.  Fix $k$ and consider
$T = \{ j< k \mid \forall x \exists y \theta (j,x,y) \}$ where $\theta$ is a $\szz$ formula.
Our goal is to prove the existence of $T$.  Suppose $( j,x,y)$ is the $n^{\text{th}}$ triple in a bijective
enumeration of $k \times \nat \times \nat$.  Define $f(n) = j$ if $y$ is the first witness that
$\forall s<x \, \exists t\le y \, \theta (j,s,t)$, and let $f(n) = k$ otherwise.  The function $f$ exists by
recursive comprehension.  For any $j<k$, $j$ appears infinitely often in the range of $f$ if and
only if $\forall x \exists y \theta (j,x,y)$.  Apply $\ect$ to $f$ and obtain a bound $b$.  Then
\[
T = \{ j<k \mid \exists x (x \ge b \land f(x) = j )\}.
\]
By bounded $\szo$ comprehension, a consequence of $\rca$ \cite{simpson}*{Theorem~II.3.9}, the set $T$ exists, proving
bounded $\pzt$ comprehension.  To complete the proof, recall that by the first step above, we may work in $\rca$.
By complementation, bounded $\pzt$ comprehension implies bounded $\szt$ comprehension.
Applying Exercise~II.3.13 of Simpson \cite{simpson}, $\szt$ induction follows from $\rca$ and bounded
$\szt$ comprehension.
\end{proof}

\subsubsection*{3.  Weihrauch analysis}

%D: Added ``to'' below
The goal of this section is to analyze $\ert$ and $\ect$ using Weihrauch reductions.
Because $\ert$ and $\ect$ have number outputs rather than set outputs, Weihrauch reducibility %D:  Why Weihrauch is the right choice
yields meaningful results where other forms of computability-theoretic reducibility would not.  %D: computability-theoretic
We will consider
Weihrauch problems defined by subsets of $\nat^\nat \times \nat$.  Each problem $P$ can be
viewed as a multifunction mapping instances $I\in {\text{domain}}(P)$ into solutions $S$ with
$(I,S)\in P$.
A problem $P$ is Weihrauch reducible to a problem $Q$, written $P\lew Q$, if instances
of $P$ can be uniformly computably transformed into instances of $Q$ whose solutions can be
uniformly computably transformed into solutions of the problem $P$.  This last transformation
may make use of the original instance of $P$.
More formally, $P \lew Q$ if there are computable functionals $\Phi$ and $\Psi$ such that
for all $I \in  {\text{domain}}(P)$, $\Phi (I) \in {\text{domain}} (Q)$, and for all $S$ such that
$(\Phi(I), S) \in Q$, %A:  Changed \Phi(S) to \Phi(I).
we have $(I, \Psi(I,S))\in P$.  We write $P \eqw Q$ if $P \lew Q$ and $Q \lew P$,
and write $P \ltw Q$ if $P\lew Q$ and $Q \not \lew P$.

The relation $\eqw$ is an equivalence relation on the Weihrauch problems.  The equivalence
%D: Added hyphen to ``well-known''
classes are called Weihrauch degrees, and many have well-known representing problems.
For example, many Weihrauch problems are known to be equivalent to the Weihrauch problem
$\lpo$ ({\emph L}imited {\emph P}rinciple of {\emph O}mniscience).  This problem
takes as an instance any $f \in \nat^\nat$, and outputs $0$ if $\exists n f(n)=0$ and $1$ otherwise.
For an introduction to Weihrauch reducibility and many Weihrauch degrees, see the article of
Brattka and Gherardi \cite{bg} and the survey
of Brattka, Gherardi, and Pauly \cite{bgp}. %A:  appended citation

Many operators on Weihrauch problems preserve reducibility.  For example,
for a problem $P$, the problem $P^{n}$ is the result of $n$ parallel applications of $P$. %A:  revised notation
The problem $P^*$ is the result of an arbitrary finite number of parallel applications of $P$.
Thus, for each $n$, we have $P^{n}\lew P^*$. %A:  deleted sup error.
Pauly introduces the concept of $P^*$ in  %D:  fixed tense error
\cite{pauly} and in Theorem~6.5 shows that $P \lew Q$ implies $P^* \lew Q^*$.  Thus
$\cdot^*$ can be viewed as an operator that preserves Weihrauch reducibility.

We may view $\ert$ as a Weihrauch problem, where the input is a number $k$ and a function $f: \nat \to k$,
and the solution is a value $b$ as provided by $\ert$, that is,
\[
\forall n \ge b \, \exists m \ge b \, (m\neq n \land f(m) = f(n)).
\]
In a similar fashion, $\ect$ can be viewed as a Weihrauch problem.  %D:  ECT as a W. problem
Our goal is to find a known Weihrauch problems equivalent to $\ert$ and to $\ect$, and
to separate $\ert$ and $\ect$ in the Weihrauch setting.  As a first step, we can state the
following theorem.

\begin{thm}\label{W1}
$\ert \eqw \lpo^*$.
\end{thm}

\begin{proof}
First we show that $\lpo^*\lew \ert$.  Given $k$ $\lpo$ instances
$f_0 , \dots f_{k-1}$ we define a coloring $g: \nat \to k+1$ as follows.
For $i<k$, let $g(nk+i) = i$  if and only if $f_i(n)=0$ and $\forall m<n (f_i (m) \neq 0)$.
Else, set $g(nk+i) = k$.  Note that by construction, all colors but $k$
appear at most once in the range of $g$.  Thus any solution to $\ert$ for $g$ must be
an upper bound for the first occurrence of $0$ in the range of any $f_i$, which
allows us to solve $\lpo$ for each $f_i$.

For the converse reduction, we can adapt the first part of the proof of Theorem~\ref{D3},
substituting $\lpo^*$ for the uses of bounded comprehension.  Given $f: \nat \to k$ we can
use finitely many parallel applications of $\lpo$ to find the non-repeating sequences of
colors in the set $S$.  Simultaneously, we can use finitely many parallel applications of
$\lpo$ to find those sequences
that appear and whose last color reappears.
Call the set of these sequences $T^\prime$.  A
sequence is in the set $T$ defined in the proof of Theorem~\ref{D3} if and only if
it is in $S$ and is not in $T^\prime$.  Given the set $T$, we can find the bound $b$
satisfying $\ert$ for $f$ by the construction in the proof of Theorem~\ref{D3}.
This shows that $\ert \lew \lpo^*$.  Summarizing, $\ert \eqw \lpo^*$.
\end{proof}

%A:  Overhauled references and exposition
Next, we turn to the Weihrauch analysis of $\ect$.  The principle
Discrete Choice, denoted $\wcn$, takes as an input an enumeration
of the complement of a nonempty set $A$ and outputs an element of $A$.
The article of Neumann and Pauly \cite{np}
introduces and studies $\wtcn$, the
total continuation of $\wcn$.
%A:  added reference
$\wtcn$ accepts the enumeration
of the complement of any set $A$, empty or not, and outputs a number,
which will be an element of $A$ if $A$ is nonempty.
Clearly, $\wcn \lew \wtcn$, and consequently $\wcn \lew \wtcn^*$.
Lemma~5 of Neumann and Pauly \cite{np} includes $\lpo^* \ltw \wcn$.
Concatenating the relations, $\lpo^* \ltw \wtcn^*$.
%Brattka and Gherardi \cite{bg} includes $\lpo \ltw \wcn$, so $\lpo \ltw \wtcn$.
%Because finite parallelization preserves ordering,
%$\lpo^* \ltw \wtcn^*$.
The next theorem links $\wtcn^*$ and $\ect$.

\begin{thm}\label{W2}
$\ect \eqw \wtcn^*$.
\end{thm}

\begin{proof}
To see that $\ect \lew \wtcn^*$, suppose the coloring $f: \nat \to k$ is an instance of $\ect$.
Our goal is to use finitely many applications of $\wtcn$ to find a value $b$ such that every color
appearing at or after $b$ appears infinitely often in the range of $f$.  For each $i<k$ construct
an enumeration of the complement of the set
\[
A_i = \{ n \mid \forall m\ge n (f(m) \neq i ) \}.
\]
Apply $\wtcn$ to each of these sets to obtain numbers $b_i$ such that if the color $i$
appears only finitely often, then it no longer appears after $b_i$.  The number
$b =1+ \max \{ b_i \mid i<k\}$ is a solution to the $\ect$ instance.

For the converse direction, suppose $A_i$ for $1\le i <k$ is a finite list of $\wtcn$ instances,
where for each $i$, $e_i$ enumerates the complement of $A_i$.  Fix a bijective pairing function
$(\cdot,\cdot): \nat\times k \to \nat$, and define the coloring $c: \nat \to k$ by
%A:  Changed \mu notation to \min.
\[
c((s,i)) =
\begin{cases}
i&{\text{if}~}i\neq 0 \land e_i(s) = \min\{ n \mid \forall t<s ( e_i(t) \neq n)\}\\
0&{\text{otherwise.}}
\end{cases}
\]
%D: Added ``\neq 0'' below
Apply $\ect$ to $c$ to find a bound $b$.  If some color $i \neq 0$ appears infinitely often in the range of $c$,
then $A_i = \emptyset$.  Otherwise, if $i$ never appears after $b$ and $s$ is sufficiently
large that $(s,i)\ge b$, then $\min\{ n \mid \forall t<s ( e_i(t) \neq n)\}$ is in $A_i$.  In either case,
$\min\{ n \mid \forall t<s ( e_i(t) \neq n)\}$ is a valid output for $\wtcn$ applied to the input $A_i$.
\end{proof}

Summarizing, we have shown that $\ert \eqw \lpo^*$, $\lpo^* \ltw \wtcn^*$, and $\wtcn^* \eqw \ect$,
so $\ert \ltw \ect$.  We have captured the strength of $\ert$ and $\ect$ in terms of known
%D: Changed ``the'' to ``than''
Weihrauch degrees, and shown that $\ert$ is strictly weaker than $\ect$ in the Weihrauch degrees.

%11212018   New sW results
Both Theorem~\ref{W1} and Theorem~\ref{W2} fail for strong Weihrauch reducibility.
In strong reducibility, the solution to the input problem must be computed from any solution of the
transformed problem without further reference to the original input.  Using the notation from the
first paragraph of this section,
$P \lesw Q$ if there are computable functionals $\Phi$ and $\Psi$ such that
for all $I \in  {\text{domain}}(P)$, $\Phi (I) \in {\text{domain}} (Q)$, and for all $S$ such that
$(\Phi(I), S) \in Q$, %A:  Changed \Phi(S) to \Phi(I).
we have $(I, \Psi(S))\in P$.  

As an example using familiar problems, we will show that $\lpo^* \ltsw \wtcn^*$.
To see that $\lpo \lesw \wtcn$, given an instance $f$ of $\lpo$, construct an
instance $g$ of $\wtcn$ by setting $g(n) = n+1$ if $f(n) \neq 0$ and $g(n) = 0$ otherwise.
If the solution for $g$ is positive, then the solution for $f$ is $0$.  If the solution for $g$ is $0$,
then the solution for $f$ is $1$.  Similarly, sequences of $\lpo$ problems can be transformed to
sequences of $\wtcn$ problems, so $\lpo^* \lesw \wtcn^*$.  We know $\wtcn^* \not\lew \lpo^*$, so
$\wtcn^* \not\lesw \lpo^*$, and thus $\lpo^* \ltsw \wtcn^*$.  The next theorem summarizes strong
reducibility relations for $\ert$ and $\ect$.

\begin{thm}\label{sW}
$\ert \ltsw \ect\ltsw \wtcn^*$, $\lpo \not\lesw \ect$, and $\ert \not\lesw \lpo^*$.
\end{thm}

\begin{proof}
Identity functionals witness $\ert \lesw \ect$.  We know $\ect \not\lew \ert$, so $\ect \not\lesw \ert$ and thus $\ert \ltsw \ect$.

The first paragraph of the proof of Theorem~\ref{W2} shows that $\ect \lesw \wtcn^*$.  The failure of the converse relation
and $\ect \ltsw \wtcn^*$ both follow from $\lpo \not \lesw \ect$, which we prove next.

To see that $\lpo \not\lesw \ect$, suppose by contradiction that $\Phi$ and $\Psi$ witness $\lpo \lesw \ect$.  Suppose
$f_1$ and $f_2$ are $\lpo$ problems with distinct solutions.  Let $\Phi(f_1 ) = g_1$ and $\Phi (f_2) = g_2$ be the associated
$\ect$ problems.  Let $b_1$ be a solution for $g_1$ and $b_2$ be a solution for $g_2$.  Then $b=\max \{b_1, b_2 \}$ is
a solution for both $g_1$ and $g_2$.  Then $\Psi (b)$ is a solution for both $f_1$ and $f_2$, yielding a contradiction.
Thus $\lpo \not \lesw \ect$.
This argument is an example of the general principle that no multifunction of the form $f:X \rightrightarrows \nat$ where all $f(x)$ are upwards closed
can compute a non-trivial multifunction $g: X  \rightrightarrows k$ for finite $k$.

To see that $\ert \not\lesw \lpo^*$, we again argue by contradiction, supposing that $\Phi$ and $\Psi$
witness $\ert \lesw \lpo^*$.
Let $f_1$ be the instance of $\ect$ consisting of a two-coloring that is constantly zero.
Suppose $\Phi (f_1) = (g_1 , \dots , g_n)$ is a sequence of $n$ instances of $\lpo$.
The computation of $\Phi (f_1)$ uses only a finite initial segment of $f_1$.  Denote the
length of this segment by $k$.  The $\lpo$ problems $g_1 , \dots , g_n$ have solutions
$s_1 , \dots , s_n$.  Thus $\Psi (s_1 , \dots , s_n )$ computes a bound $m$
satisfying $\ert$ for $f_1$.  Now consider the $\ert$ problem $f_2$, consisting of a two-coloring
that contains $k+m$ zeros, followed by a single one, followed by an infinite string of zeros.
Because $f_2$ and $f_1$ agree on the first $k$ values, $\Phi (f_2)= \Phi(f_1)= (g_1 , \dots , g_n)$.
These $\lpo$ problems are the same as before, and so still have the solutions $s_1 , \dots , s_n$.
Thus $\Psi (s_1 , \dots , s_n) = m$ should be a bound satisfying $\ert$ for $f_2$.  However,
by the construction of $f_2$, any bound for $f_2$ must be at least $k+m+1$, which is strictly
larger than $m$.  Thus $\Phi$ and $\Psi$ cannot be witnesses of $\ert \lesw \lpo^*$, and we
have shown that $\ert \not \lesw \lpo^*$.
\end{proof}

Minor alterations in the formulations of $\ert$ and $\ect$ can result in interesting variations in their
Weihrauch strengths.  For example, let $\minert$ denote the principle that outputs the minimum bound
satisfying $\ert$.  Define $\minect$ similarly.

\begin{thm}\label{minert}
$\ert \eqw \minert$ 
and $\rcas$ proves $\ert \leftrightarrow \minert$.
\end{thm}

\begin{proof}
Every solution of $\minert$ is a solution of $\ert$, so $\ert \lew \minert$.
For the converse, apply the second paragraph of the proof of Theorem~\ref{W1}, using $\lpo^*$ to find the set $T$.
By the note following the proof of Theorem~\ref{D3}, $T$ can be used to calculate the minimum bound.
Thus $\minert \lew \lpo^*$.  By Theorem~\ref{W1},
$\lpo^* \lew \ert$, so $\minert \lew \ert$.

For the reverse mathematics result, $\rcas$ proves $\minert$ implies $\ert$ trivially.
To prove the converse, assume $\ert$ and let $f:\nat \to k$.
By $\ert$, we can find a bound $b$.  By Theorem~\ref{D3}, $\ert$ implies $\Sigma^0_1$ induction,
so by Lemma~\ref{D1} we can use bounded $\Sigma^0_1$ comprehension to find
$Y = \{ c<k \mid \exists x ( x \ge b \land f(x) = c\}$, the range of $f$ on $[b,\infty)$.  By the $\Sigma^0_0$
least element principle, there is a least $n\le b$ such that for all $t\in [n,b]$, either
$f(t)\in Y$ or $f(t)$ appears at least twice in $[n,b]$.  This least $n$ satisfies $\minert$.
\end{proof}

In contrast to Theorem~\ref{minert}, we will prove below that $\ect \ltw \minect$.
Our proof uses the following characterization of $\minect$ in terms of $\wtcn$ and $\isinf$.
The principle $\isinf$ takes an infinite binary string as input, outputs $1$ if it has infinitely many
ones, and outputs $0$ otherwise.  The notation $P\times Q$ denotes the principle corresponding
to solving $P$ and $Q$ in parallel.

\begin{thm}\label{minect}
$\minect \eqw \wtcn^* \times \isinf^*$.
\end{thm}

\begin{proof}
To see that $\minect \lew \wtcn^* \times \isinf^*$, let $f : \nat \to k$ be an instance of $\minect$.
For each $j<k$, we can use one instance of $\isinf$ to determine if $j$ appears infinitely many times
in the range of $f$, and one instance of $\wtcn$ to find the last occurrence of $j$ in the case that $j$
appears only finitely many times.  Adding one to the maximum of the positions for the values that do not appear
infinitely many times yields the desired output for $\minect$.

The converse relation takes a few steps.  By Theorem~\ref{W2}, $\wtcn^* \eqw \ect$.  Trivially,
$\ect \lew \minect$, so $\wtcn^* \lew \minect$.

To see that $\isinf \lew \minect$, let $p$ denote an infinite binary sequence.  Let $r$ be the sequence consisting
of a $1$ followed by the result of alternating $0$ with digits from $p$.  Then $\minect (r)$ is $0$ if and only if
$1$ appears infinitely many times in $p$.

Next, we show that $\minect$ is idempotent, or to be more precise, that $\minect \times \minect \lew \minect$.  Let
$\langle \cdot , \cdot \rangle : \nat \times \nat \to \nat$ be a bijective map
such that if $m_0 \le m_1$ and $n_0 \le n_1$, then $\langle  m_0 , n_0 \rangle \le \langle m_1, n_1 \rangle$.
Let $p$ and $q$ be instances
of $\minect$.  Replace $p(0)$ with a color not appearing in the range of $p$.  This increases the value of
$\minect$ by one only in the case that every color appears infinitely often in the original sequence.
We can now assume that at least one color appears only finitely many times in $p$. 
Make the same adjustment and assumption for $q$.  Define
the coloring $r$ by $r(\langle n,m \rangle ) = \langle p(n),q(m)\rangle$.  If $n_0$ is the last time some color
$c_0$ appears in $p$, and $n_1$ is the last time that some color $c_1$ appears in $q$, then
$\langle n_0 , n_1 \rangle$ is the last time that $\langle c_0 , c_1 \rangle$ appears in $r$.  Conversely, if $\langle c_0, c_1\rangle$ appears for
the last time at position $\langle n_0 , n_1\rangle$, then $c_0$ must appear last in $p$ at $n_0$, and $c_1$ must
appear last in $q$ at $n_1$.  Thus, solutions for $p$ and $q$ can be extracted from the solution for $r$.

Iterated applications of the idempotence of $\minect$ (or an application of Proposition 4.4 of \cite{bgp}) show that
$\minect^* \lew \minect$.  Because $\isinf \lew \minect$, we have $\isinf^* \lew \minect^* \lew \minect$.  We have
already shown that $\wtcn^* \lew \minect$, so
$\wtcn^* \times \isinf^* \lew \minect \times \minect \lew \minect$, completing the proof of the theorem.
\end{proof}

The next result assists in separating $\ect$ and $\minect$.

\begin{thm}\label{isitc}
$\isinf \not \lew \wtcn^*$.
\end{thm}

%Begin 12/13 corrections for the proof of Theorem 13
\begin{proof}
Suppose by way of contradiction that $\Phi$ and $\Psi$ witness $\isinf  \lew \wtcn^*$.
The function mapping sequences $p$ to the number of instances of $\wtcn$ in $\Phi (p)$ is
computable and therefore continuous.  Let $\sigma_0$ and $n$ be such that $\Phi (p)$ consists
of $n$ instances of $\wtcn$ for all $p \succcurlyeq\sigma_0$, that is for all $p$ extending $\sigma_0$.
Denote the ranges of these instances by $\Phi(p)_1 , \dots , \Phi(p)_n$.
For $i \le n$ and $m \in \nat$, define $C_{m,i} = \{\sigma \succcurlyeq \sigma_0 \mid m \in \Phi(\sigma )_i \}$.
Let $\sigma_1$ be an extension of $\sigma_0$ such that for each $i \in [1,n]$, either every
$C_{m,i}$ is dense below $\sigma _1$, or there is an $m_i$ such that $C_{m_i, i}$ contains no extension
of $\sigma_1$.  Let $F$ be the set of all $i$ such that $m_i$ is defined.

Let $p$ consist of $\sigma_1$ followed by an infinite sequence of zeros.  The sequence $p$ has finitely many ones.
There is a solution $(a_1 , \dots , a_n )$ of $\Phi (p)$ such that $a_i = m_i$ for all $i \in F$.
Then $\Psi (a_1 , \dots , a_n , p)$ returns $0$, with a computation that depends only on $a_1 , \dots, a_n$
and a finite initial segment $\sigma_2$ of $p$.
Let $g \succ \sigma_2$ be 1-generic.
If $i \notin F$, then for every $m$, $C_{m , i}$ is dense below $\sigma_2$, so $\Phi(g)_i = \nat$.
Thus, $(a_1, \dots , a_n )$ is a solution of $\Phi (g)$.  But
$\Psi (a_1 , \dots , a_n , g) = \Psi (a_1 , \dots , a_n , p ) = 0$ and $g$ has infinitely many ones,
yielding the desired contradiction.
\end{proof}
%End 12 13 corrections for the proof of Theorem 13

\begin{thm}\label{penul}  %Note:  coro changed to thm
$\ect \ltw \minect$ and $\rcas$ proves $\ect \leftrightarrow \minect$.
\end{thm}

\begin{proof}
Trivially, $\ect \lew \minect$.
To prove the strict inequality, suppose by contradiction that $\minect \lew \ect$.  By Theorem~\ref{minect}, $\isinf \lew \minect$, so by
Theorem~\ref{W2}, $\isinf \lew \wtcn^*$, contradicting Theorem~\ref{isitc}.

Shifting focus to reverse mathematics, trivially $\rca^*$ proves that $\minect$ implies $\ect$.
%Begin 12/13 corrections for Thm 14
%{\color{red}\sout{Furthermore, by}}
%{\color{blue} 
For the converse, assuming $\ect$, by Theorem~\ref{D4}, we may use $\Sigma^0_2$ induction.
By
the $\Pi^0_1$ least element principle (a consequence of $\Sigma^0_1$ induction), a minimal
bound can be found in the first part of the proof of Theorem~\ref{D4}.  Thus, over $\rca^*$, $\ect$
is equivalent to $\minect$.
%End 12/13 correction for Thm 14
\end{proof}

Theorem~\ref{penul} demonstrates the ability
of Weihrauch reductions to make finer distinctions in this setting.

Our final result links $\minect$ to principles considered by Hirst and Mummert \cite{hm}.
The principle $\cmh$ takes as inputs a size $n$ and the enumeration of the complement of a collection of finite
subsets of $\nat$, each of size at most $n$, and outputs an element of the collection of maximum cardinality.

\begin{thm}\label{last}
$\minect \eqw \cmh$.
\end{thm}

\begin{proof}
From Theorem~\ref{minect} we know $\minect \eqw \wtcn^* \times \isinf^*$, so it suffices to show that
\[
\minect \lew \cmh \lew \wtcn^* \times \isinf^*.
\]
For the first reduction, suppose $f: \nat \to k$ is an instance of $\minect$.
Consider the set $A$ of all finite sets $F \subset k \times \nat$ such that for each $j<k$,
if $(j,n) \in F$ then $n$ is the maximum natural number such that $f(j) = n$.  Here we are
identifying pairs with their integer codes, so $F$ can be viewed as a subset of $\nat$.
The set $A$ is $\Pi^0_1$ definable using $f$ as a parameter, and its complement can
be enumerated by a function uniformly computable from $f$.  Use this enumeration
and the size $k$ as the input for $\cmh$, and let $F_0$ be the resulting maximal
output set.  Adding $1$ to the maximum of the second coordinates of the elements of $F_0$
yields the desired bound for $\minect$.

To prove the final reduction, it is useful to note that $\wtcn$ can be used to count the numbers
of ones in a binary string.  Using a bijective pairing function, given a sequence $p : \nat \to 2$,
we can define an enumeration $q$ of the (codes for) pairs that omits at most one pair, so that
the first coordinate of that omitted pair corresponds to the number of ones in the range of $p$, provided that number is finite.
Calculation of $q$ can be viewed as a moving marker process.  Place a marker on $(0,0)$ and
then alternate enumerating unmarked pairs and calculating values of $p$ until a $1$ appears in the range of $p$.  Move
the marker to the first non-enumerated pair with an initial coordinate of $1$, enumerate $(0,0)$,
and continue enumerating unmarked pairs and calculating values of $p$ until the next $1$ appears in the range of $p$.
Iterate.  If there
are infinitely many ones in the range of $p$, then $q$ will enumerate all pairs.  If only finitely
many ones appear, $\wtcn$ applied to $q$ will find a pair with the desired first coordinate.

To prove that $\cmh\lew \wtcn^* \times \isinf^*$, let $f : \nat \to \nat^{<\nat}$ enumerate the complement
of a set $A$ of finite subsets of $\nat$, each of size less than $k$.  For each positive $i<k$, let
$e_i$ be an enumeration of all the subsets of $\nat$ of size exactly $i$.  For each positive $i<k$, define
the instance $p_i$ of $\isinf$ as follows.
Set $p_i (n) = 1$ if there is a $t \le n$ such that $f(t) = e_i (c_t )$ where
$c_t = |\{j<n \mid p_i (j) = 1\}|$, and set $p_i (n)=0$ otherwise.  Thus $f$ enumerates all sets of size $i$ if and only if the range of $p_i$ contains
infinitely many ones, and the range of $p_i$ contains a total of $n$ ones if and only if $e_i (n)$ is the
first set enumerated by $e_i$ that is in $A$.  For each $p_i$, let $q_i$ be the associated instance
of $\wtcn$ that counts the ones in the range of $p_i$.  Given the solutions to $\isinf$ for each $p_i$ and
to $\wtcn$ for each $q_i$ for all $i<k$, we can find the maximum $j$ such that $\isinf$ fails for $p_j$.
If $n$ is the output from $\wtcn$ for $q_j$, then $e_j (n)$ is a maximal element of $A$, solving the
instance $\cmh$ corresponding to $f$.
\end{proof}

Hirst and Mummert \cite{hm} showed that $\cmh$ is Weihrauch equivalent and provably equivalent over $\rca$
to several principles formalizing calculation of bases for bounded dimension matroids and vector spaces,
and finding connected component decompositions of graphs with finitely many components.  Thus $\minect$
is Weihrauch equivalent to all these principles, $\ect$ is strictly Weihrauch weaker, and all of them are provably equivalent
over $\rca$.

\subsubsection*{Acknowledgements}

Discussions contributing to this work occurred at the Workshop on Ramsey Theory and Computability,
organized by Peter Cholak and held July 9-13 of 2018 at the Rome Notre Dame Global Gateway.
Further discussions occurred at Dagstuhl
Seminar 18361,
organized by Vasco Brattka, Damir Dzhafarov, Alberto Marcone, and Arno Pauly,
and held September 2-7 of 2018 at Schloss Dagstuhl,
the Leibniz-Zentrum f\"u{}r Informatik.
Denis Hirschfeldt was partially supported by grant DMS-1600543 from the National Science Foundation of the United States.
Jeffry Hirst's travel was supported by
a SAFE grant from the College of Arts and Sciences, a grant from the Office of International
Education and Development, and a Board of Trustees travel grant, all from
Appalachian State University.
Yokoyama was partially supported by JSPS KAKENHI (grant numbers
16K17640 and 15H03634), JSPS Core-to-Core Program
(A.~Advanced Research Networks) and JAIST Research Grant 2018 (Houga).

\newpage

%author information
\begin{list}{}{\setlength{\leftmargin}{0pt}}
\item
\begin{list}{}{\setlength\itemsep{0pt}\setlength\parsep{0pt}}
\item Caleb Davis
\item Department of Mathematical Sciences, Appalachian State University
\item ASU Box 32092, Boone, NC 28608  USA
\item davisc3@appstate.edu
\end{list}
\item
\begin{list}{}{\setlength\itemsep{0pt}\setlength\parsep{0pt}}
\item Denis R. Hirschfeldt
\item Department of Mathematics, The University of Chicago
\item 5734 S. University Ave., Chicago, IL 60637 USA
\item drh@math.uchicago.edu
\end{list}
\item
\begin{list}{}{\setlength\itemsep{0pt}\setlength\parsep{0pt}}
\item Jeffry Hirst
\item Department of Mathematical Sciences, Appalachian State University
\item ASU Box 32092, Boone, NC 28608 USA
\item hirstjl@appstate.edu
\end{list}
\item
\begin{list}{}{\setlength\itemsep{0pt}\setlength\parsep{0pt}}
\item Jake Pardo
\item Department of Mathematical Sciences, Appalachian State University
\item ASU Box 32092, Boone, NC 28608 USA
\item jpardo@alumni.nd.edu
\end{list}
\item
\begin{list}{}{\setlength\itemsep{0pt}\setlength\parsep{0pt}}
\item Arno Pauly
\item Department of Computer Science, Swansea University
\item Singleton Park, Swansea SA2 8PP  United Kingdom
\item arno.m.pauly@gmail.com
\end{list}
\item
\begin{list}{}{\setlength\itemsep{0pt}\setlength\parsep{0pt}}
\item Keita Yokoyama
\item School of Information Science
\item Japan Advanced Institute of Science and Technology
\item 1-1 Asahidai, Nomi, Ishikawa 923-1292 Japan
\item y-keita@jaist.ac.jp
\end{list}
\end{list}

\end{document}